\documentclass[article,12pt]{amsart}
\usepackage{lineno,hyperref,enumitem}
\usepackage{amsmath,amssymb,dsfont}

\newtheorem{thm}{Theorem}[section]
\newtheorem{defn}[thm]{Definition}
\newtheorem{prop}[thm]{Proposition}

\newtheorem{rem}[thm]{Remark}

\newtheorem{ex}[thm]{Example}

\begin{document}

\title[Study of weaving frames in Krein spaces]{Study of weaving frames in Krein spaces}



\author[A. Bhardwaj]{Avinash Bhardwaj}

\author[A. Bhandari]{Animesh Bhandari$^*$}
\address{Department of Mathematics\\ SRM University, {\it AP} - Andhra Pradesh\\ Neeru Konda, Amaravati, Andhra Pradesh - 522502\\ India}

\email{bhandari.animesh@gmail.com, animesh.b@srmap.edu.in}

\keywords{frame, weaving frames, fundamental symmetry, Krein space\\ $*$ is the corresponding author}

\begin{abstract}
Inspired by the work of Bemrose et al. \cite{Be16}, we delve into the study of weaving frames in Krein spaces. This paper presents a comprehensive exploration of various properties and characterizations of Krein space weaving frames. In support of our findings, several examples and counter examples are provided, illustrating the applicability of the theoretical results. Additionally, we extend the discussion to an important application in probabilistic erasure, highlighting how weaving frames can be used to mitigate data loss in such scenarios. This work contributes to the broader understanding of frame theory in indefinite inner product spaces.

\end{abstract}

\maketitle



\section{Introduction}
The concept of frames, first introduced by Duffin et al. \cite{duffin1952class}, gained widespread attention when Daubechies et al. \cite{2} underscored its importance in signal processing. This renewed interest led to frame theory becoming a crucial tool in various fields, including Mathematics, Physics, and Engineering. Its versatility has since made frames indispensable in different areas, such as, wavelet analysis, quantum mechanics, and communication systems.

We first discuss the basics of frames to establish the foundational concepts required to our results.

\begin{defn}
A sequence  $\{ f_n \}_{n=1}^\infty$ in a Hilbert space $\mathcal{H}$ is called a \emph{frame} for $\mathcal H$ if there exist constants $A,B >0$ so that 
\begin{equation}\label{Eq:Frame} A\|f\|^2~ \leq ~\sum\limits_{n=1}^\infty |\langle f,f_n\rangle|^2 ~\leq ~B\|f\|^2,
\end{equation} for all $f \in \mathcal{H}$. The numbers $A, B$ are called \emph{frame bounds}. For a given  frame $\{f_n\}_{n=1}^\infty$ for $\mathcal{H}$, the \emph{pre-frame operator} or \emph{synthesis operator} is a bounded linear operator $T: \ell^2(\mathbb N)\rightarrow \mathcal{H}$ and is defined as $T\{c_n\}_{n=1}^\infty = \sum\limits_{n=1}^\infty c_n f_n$. The adjoint of $T$, $T^*: \mathcal{H} \rightarrow \ell^2(\mathbb N)$ is  given by $T^*f = \{\langle f, f_n\rangle\}_{n=1}^\infty$, is called the \emph{analysis operator}. The \emph{frame operator}, $S=TT^*: \mathcal{H}\rightarrow \mathcal{H}$, is defined as 
\begin{equation}
Sf=TT^*f = \sum\limits_{n=1}^\infty \langle f, f_n\rangle f_n~~\emph{for~all~} f\in {\mathcal H}.
\end{equation}
We notice that depending upon the nature of frame bounds we get special kinds of frames, e.g. in the inequality (\ref{Eq:Frame}) if $A=B$ then the given frame is called a tight frame, in addition to the same if $A=B=1$ it is called a Parseval frame. It is well-known that the frame operator is bounded, positive, self adjoint and invertible.

Using the invertibility property, every element in the Hilbert space can be reconstructed as,
\begin{equation}
f=\sum\limits_{n=1}^\infty \langle f, S^{-1}f_n \rangle f_n = \sum\limits_{n=1}^\infty \langle f, f_n \rangle S^{-1}f_n.
\end{equation}

Furthermore, $T^*T\{c_n\}=\sum\limits_{m=1}^\infty \langle f_n, f_m\rangle c_n$ for every $\{c_n\} \in \ell^2(\mathbb N)$ and its matrix representation is $(T^*T)_{m,n} = \langle \delta_m, T^*T \delta_n \rangle_{l^2}$. Thus we have,
\begin{eqnarray*}
(T^*T)_{m,n} = \langle \delta_m, T^*T \delta_n \rangle_{l^2} &=& \langle T \delta_m, T \delta_n \rangle_{\mathcal H}
\\&=& \langle f_m, f_n \rangle_{\mathcal H}
\\&=& G_{m, n}
\end{eqnarray*}
which is known as Gramian matrix. It is to be noted that the Gramian of a frame is invertible if and only if the given frame is a Riesz basis.

For detailed discussion regarding frames, their generalization and applications we refer \cite{christensen2003introduction, 1, Kaf09, Gu18, 3, Po20, Ba18, Fi97, Fi98, Fi06, Fi00, Na12, Na13, Bo06, Cz02, Cz06, Ha07, Ha00, He04, An79}.
\end{defn}

If $(\mathcal K, [.,.])$ is a Krein space with fundamental decomposition $\mathcal K_{\delta^+} \oplus \mathcal K_{\delta^-}$ and  fundamental  symmetry  $J = P_{\delta^+} - P_{\delta^-}$, where $P_{\delta^{\pm}}$ is the orthogonal  projection onto $\mathcal{K}_{\delta^{\pm}}$, then for every $k \in \mathcal K$ we have, $$J(k) = J(k_{\delta^+}+ k_{\delta^-}) = k_{\delta^+} - k_{\delta^-}, \quad k = k_{\delta^+}+ k_{\delta^-}\in \mathcal K_{\delta^+} \oplus \mathcal K_{\delta^-}. $$ In this context the associated $J$-inner product is given by,
$$[k, k']_J = [k, Jk'],  \quad k, k' \in \mathcal{K},$$ and hence ($\mathcal{K}, [., .]_J$) is a Hilbert space. The inner product $[., .],$ which is positive definite, induces a topology on Krein space $\mathcal{K}$ through $J$-norm define as, $$\|k\|_{J}^2 = [k, k]_J = [k, Jk], \quad \forall k \in \mathcal{K}.$$
Here ($\mathcal{K}_{\delta^+}, +[., .]$) is a Hilbert space and ($\mathcal{K}_{\delta^-}, -[., .]$)  is known as anti Hilbert space.

A $J$-orthonormalized system is a family of vectors $\{e_n\}_{n=1}^\infty$ satisfying condition $[e_m, e_n]=\delta_{mn}$ for every $m, n \in \mathbb N$. A $J$-orthonormal basis is a $J$-orthonormalized system that also serves as a Schauder basis, providing unique and stable expansions for vectors. If $\{e_n\}_{n=1}^\infty$ is a $J$-orthonormal basis in $\mathcal K$, then any element $k \in \mathcal K$ can be represented as $k=\sum\limits_{n=1}^\infty \lambda_n [k, e_n] e_n$, where $\lambda_n=[e_n, e_n]=\pm 1$.

A sequence $\{z_n\}_{n=1}^\infty$ in a Krein space $(\mathcal K, [.,.])$ is said to be a frame for $\mathcal K$, if for every $k \in \mathcal K$ there exist positive constants $\alpha, \beta$ we have,
\begin{equation}
\alpha \|k\|_{J}^2 \leq \sum\limits_{n=1}^\infty |[k, z_n]|^2 \leq \beta \|k\|_{J}^2,
\end{equation}
where $\|\|_{J}$ is the standard norm in $(\mathcal K, \langle .,. \rangle)$. It is to be noted that a sequence $\{z_n\}_{n=1}^\infty$ is a frame for a Krein space $(\mathcal K, [.,.])$ if and only if it is a frame for $(\mathcal K, \langle .,. \rangle)$ (see \cite{Ga}). The key difference between $J$-frames and classical frames for Hilbert spaces is that some frames in Hilbert space $(\mathcal K, \langle .,. \rangle)$  are not necessarily $J$-frames in Krein spaces $(\mathcal K, [.,.])$. For detailed discussion regarding the same we refer (Example 3.3, \cite{Gi12}).\\





\section{Main Results}

In this section, we explore various properties and characterizations of weaving frames within the structure of Krein spaces. We begin by introducing the concept of Krein space weaving frames, which extends the classical frame theory into indefinite inner product spaces. Our discussion focuses on  the interplay between the positive and negative components of Krein spaces.

\begin{defn}
Let $\{z_n\}_{n=1}^\infty$ and $\{z'_n\}_{n=1}^\infty$ be two frames for the Krein space $(\mathcal K, [.,.])$, then they are said to be weaving frames for $\mathcal K$ if for for every $\sigma \subset \mathbb N$, $\{z_n\}_{n \in \sigma} \cup \{z'_n\}_{n \in \sigma^c}$ is a frame for $\mathcal K$. Thus for every $k \in \mathcal K$ and every $\sigma \subset \mathbb N$, there exist positive constants $\alpha \leq \beta$ we have,
\begin{equation}\label{1}
\alpha \|k\|_{J}^2 \leq \sum\limits_{n \in \sigma} |[k, z_n]|^2 + \sum\limits_{n \in \sigma^c} |[k, z'_n]|^2 \leq \beta \|k\|_{J}^2.
\end{equation}
\end{defn}

We now define the corresponding  frame operator for the associated weaving frames, establishing its roles in the framework.

The weaving frame operator $S:\mathcal K \rightarrow \mathcal K$ is defined as, 
$$S_{\sigma}(k)=TJT^*=\sum\limits_{n \in \sigma} [k, z_n] z_n + \sum\limits_{n \in \sigma^c} [k, z'_n] z'_n,$$ where $T, T^*$ and $J$ are associated synthesis, analysis operators and fundamental symmetry on $(\mathcal K, [.,.])$.

\begin{ex}
let us assume a Krein space $(\mathcal K = \mathbb C^2, [.,.])$ with the associated indefinite inner product defined as $[(k_1 , k_2), (l_1,  l_2)] = k_1\overline l_1 - k_2\overline l_2$, where $\ (k_1, k_2)  \in {\mathcal K}$, $\ (l_1, l_2) \in {\mathcal K}$. Consider two frames 
$$
\{z'_n\}_{n=1}^3\ = \{(1, 1), (0, 1), (1, 0) \}$$ and $$\{z''_n\}_{n= 1}^3\ = \{(-1, -1), (0, -1), (1, 0) \}.
$$ Then it is easy to verify that for every $\sigma \subset \{1, 2, 3\}$, the equation~\eqref{1} is satisfied. Therefore $\{z'_n\}_{n = 1}^3$ and $\{z''_n\}_{n = 1}^3$ are weaving frames for $(\mathcal K = \mathbb C^2, [.,.])$.

On the other hand , if we consider a Krein space  $(\mathcal K = \mathbb C^3, [.,.])$. with the associated indefinite inner product $[. , .]$ defined as 
$$
[ (k_1, k_2, k_3 ), (l_1, l_2, l_3 )] = k_1\overline l_1 + k_2\overline l_2 - k_3\overline l_3,
$$ where $\ (k_1, k_2, k_3) \in {\mathcal K}$, $\ (l_1, l_2, l_3) \in {\mathcal K}$. Let us consider two frames 
$$
\{z'_n\}_{n=1}^3\ = \{(1, 2, 0), (0, \sqrt{2}, 0), (0, 0, 1) \}
$$ and 
$$
\{z''_n\}_{n= 1}^3\ = \{(0, -\sqrt{2}, 0), (-2, 0, 0), (0, 0, 2) \}.
$$
 It is easy to verify that these frames are not weaving frames if we consider $\sigma$ = $\{2\}$, since for $k =(k_1, 0, 0)$ with $k_1 \neq 0$ we have  $\|k\|_{J}^2 =  \|(k_1, 0,0)\|_{J}^2 = |k_1|^2$, where as 
\begin{eqnarray*}
&&\sum\limits_{n \in \sigma} |[k, z'_n]|^2 + \sum\limits_{n \in \sigma^c} |[k, z''_n]|^2 \\&=& |[(k_1, 0, 0), (0, \sqrt{2}, 0)]|^2
+|[(k_1, 0, 0), (0, -\sqrt{2}, 0)]|^2
 + |[(k_1, 0, 0), (0, 0, 2)]|^2 
\\&=& 0.
\end{eqnarray*}
 Thus the left hand inequality of equation (\ref{1}) is not satisfied.
Therefore, the given frames are not weaving frames.

\end{ex}We present both an example and a non-example of Krein space weaving frames to illustrate their defining characteristics and limitations.

Our first result provides a characterization of weaving frames in Krein spaces through the use of the associated fundamental symmetry.

\begin{thm}\label{char}
Let $\{z_n\}_{n=1}^\infty$ and $\{z'_n\}_{n=1}^\infty$ be two frames for the Krein space $(\mathcal K, [.,.])$, then the following are equivalent:

\begin{enumerate}
\item $\{z_n\}_{n=1}^\infty$ and $\{z'_n\}_{n=1}^\infty$ are weaving frames for $\mathcal K$.
\item $\{Jz_n\}_{n=1}^\infty$ and $\{Jz'_n\}_{n=1}^\infty$ are weaving frames for $\mathcal K$.
\item $\{z_n\}_{n=1}^\infty$ and $\{z'_n\}_{n=1}^\infty$ are weaving frames for $(\mathcal K, [.,.]_J)$.
\item $\{Jz_n\}_{n=1}^\infty$ and $\{Jz'_n\}_{n=1}^\infty$ are weaving frames for $(\mathcal K, [.,.]_J)$.
\end{enumerate}
\end{thm}

The validity of the above theorem is demonstrated through the following example.

\begin{ex}
Consider $\mathcal K = \ell^2$ and the indefinite inner product on $\ell^2$ is defined as
$$
[k,k']=\sum\limits_{n=1}^\infty (-1)^n z_n \overline{z'_n},
$$ where $k=\{z_n\}_{n=1}^\infty, k'=\{z'_n\}_{n=1}^\infty \in \ell^2$. Suppose that $K_{\delta^+} = \{\{z_n\}_{n=1}^\infty: z_n=0 \text{~if ~n ~is~ odd}\}$ and $K_{\delta^-} = \{\{z'_n\}_{n=1}^\infty: z_n=0 \text{~if ~n ~is~ even}\}$. Then $\mathcal K = K_{\delta^+} \oplus  K_{\delta^-}$, where $K_{\delta^+}$ and $K_{\delta^-}$ are complete and hence $\mathcal K = \ell^2, [.,.]$ is a Krein space. 

Consider the frames 
$$
\{z_n\}_{n=1}^\infty=\{e_1, 0, e_2, e_2, e_3, e_3, \cdots\}
$$ and 
$$
\{z'_n\}_{n=1}^\infty=\{e_1, e_1, e_2, e_2, e_3, e_3, \cdots\}
$$ for $\ell^2$, where $\{e_n\}$ is the standard orthonormal basis in $\ell^2$. Then it is easy to see that they are weaving frames for $\ell^2$. If 
$$
J=\begin{pmatrix}
-1 & 0 & 0 & 0 & 0 & \cdots\\
0 & 1 & 0 & 0 & 0 & \cdots\\
0 & 0 & -1 & 0 & 0 & \cdots\\
0 & 0 & 0 & 1 & 0 & \cdots\\
 \vdots & \vdots & \vdots & \vdots & \vdots & \ddots
\end{pmatrix},
$$

Then $\{Jz_n\}_{n=1}^\infty = \{-e_1, 0, e_2, e_2, -e_3, -e_3, \cdots\}$ and \\
$\{Jz'_n\}_{n=1}^\infty = \{-e_1, -e_1, e_2, e_2, -e_3, -e_3, \cdots\}$, and therefore it is easy to verify that the remaining conditions of the theorem are satisfied.
\end{ex}

\noindent \underline{\bf Proof of Theorem 2.3:}\\
\noindent  (\underline{1 $\Longleftrightarrow$ 4}) Let $\{z_n\}_{n=1}^\infty$ and $\{z'_n\}_{n=1}^\infty$ be weaving frames for $\mathcal K$ with the universal bounds $\alpha, \beta$. Then for every $k \in \mathcal K$ and every $\sigma \subset \mathbb N$ we have, 
\begin{equation}\label{5}
	\alpha \|k\|_{J}^2 \leq \sum\limits_{n \in \sigma} |[k, z_n]|^2 + \sum\limits_{n \in \sigma^c} |[k, z'_n]|^2 \leq \beta \|k\|_{J}^2.
\end{equation}
Now we prove $J^2=\mathcal I$ and $J=J^*$, where $\mathcal I$ is the identity operator and $J$ is the fundamental symmetry on $\mathcal K$. 

Since $\mathcal K = K_{\delta^+} \oplus  K_{\delta^-}$, then for every $k \in \mathcal K$ can be written as $k=k_{\delta^+} + k_{\delta^-}$ and $J = \pi^+ - \pi^-$, where $ \pi^+$ is the orthogonal projection onto $K_{\delta^+}$ and $ \pi^-$ is the orthogonal projection onto $K_{\delta^-}$. 

Thus for every $k \in \mathcal K$ we have,
\begin{eqnarray*}
	J^2(k)=J(Jk)=J(J(k_{\delta^+} + k_{\delta^-}))&=&J( (\pi^+ - \pi^-)(k_{\delta^+} + k_{\delta^-}))\\
	&=&J(k_{\delta^+} - k_{\delta^-})\\
	&=&  (\pi^+ - \pi^-)(k_{\delta^+} - k_{\delta^-})\\
	&=& k_{\delta^+} + k_{\delta^-})\\
	&=& k.
\end{eqnarray*}
Furthermore, for every $k \in \mathcal K$ we have,
\begin{eqnarray*}
	[J^*(k), k] = [k, J(k)]=[k, (\pi^+ - \pi^-)(k)] &=& [(\pi^+ - \pi^-)^*(k), k]\\
	&=& [(\pi^+ - \pi^-)(k), k]\\
	&=& [J(k), k].
\end{eqnarray*}
Therefore, applying equation (\ref{5}), for every $k \in \mathcal K$ and every $\sigma \subset N$ we have,
\begin{equation*}
	\alpha \|k\|_{J}^2 \leq \sum\limits_{n \in \sigma} |[Jk, Jz_n]|^2 + \sum\limits_{n \in \sigma^c} |[Jk, Jz'_n]|^2 \leq \beta \|k\|_{J}^2.
\end{equation*}
and hence we obtain,
\begin{equation*}
	\alpha \|k\|_{J}^2 \leq \sum\limits_{n \in \sigma} |[k, Jz_n]_{J}|^2 + \sum\limits_{n \in \sigma^c} |[k, Jz'_n]_{J}|^2 \leq \beta \|k\|_{J}^2.
\end{equation*}

\noindent  (\underline{2 $\Longleftrightarrow$ 3}) Let $\{J z_n\}_{n=1}^\infty$ and $\{J z'_n\}_{n=1}^\infty$ be weaving frames for $\mathcal K$ with the universal bounds $\alpha, \beta$. Then for every $k \in \mathcal K$ and every $\sigma \subset \mathbb N$ we have, 
\begin{eqnarray*}
	&& \alpha \|k\|_{J}^2 \leq \sum\limits_{n \in \sigma} |[k, J z_n]|^2 + \sum\limits_{n \in \sigma^c} |[k, J z'_n]|^2 \leq \beta \|k\|_{J}^2.\\
	& \Leftrightarrow & \alpha \|k\|_{J}^2 \leq \sum\limits_{n \in \sigma} |[k, z_n]_{J}|^2 + \sum\limits_{n \in \sigma^c} |[k, z'_n]_{J}|^2 \leq \beta \|k\|_{J}^2.
\end{eqnarray*}
Thus $\{z_n\}_{n=1}^\infty$ and $\{z'_n\}_{n=1}^\infty$ are weaving frames for $(\mathcal K, [.,.]_J)$ and conversely.

\noindent  (\underline{3 $\Longleftrightarrow$ 4}) Let $\{z_n\}_{n=1}^\infty$ and $\{z'_n\}_{n=1}^\infty$ are weaving frames for $(\mathcal K, [.,.]_J)$ with the universal bounds $\alpha, \beta$. 

Then for every $k \in \mathcal K$ and every $\sigma \subset \mathbb N$ we have, 
\begin{equation}\label{6}
	\alpha \|k\|_{J}^2 \leq \sum\limits_{n \in \sigma} |[k, z_n]_{J}|^2 + \sum\limits_{n \in \sigma^c} |[k, z'_n]_{J}|^2 \leq \beta \|k\|_{J}^2.
\end{equation}
For every $k \in \mathcal K$, we prove $\|J(k)\|_{J}=\|k\|_{J}$. Since $\|J\|_{J}=1$, then $\|J(k)\|_{J} \leq \|k\|_{J}$. 

Furthermore, for every $k \in \mathcal K$, $\|k\|_{J} \leq \|J^{-1}\|_{J} \|J(k)\|_{J}$, and hence we have $\|k\|_{J} \leq \|J(k)\|_{J}$.  Thus for every $k \in \mathcal K$, $\|J(k)\|_{J}=\|k\|_{J}$.

Therefore, using equation (\ref{6}) and replacing $k$ by $J(k)$, for every every $\sigma \subset \mathbb N$ we obtain,
\begin{eqnarray*}
	&& \alpha \|J(k)\|_{J}^2 \leq \sum\limits_{n \in \sigma} |[J(k), z_n]_{J}|^2 + \sum\limits_{n \in \sigma^c} |[J(k), z'_n]_{J}|^2 \leq \beta \|J(k)\|_{J}^2\\
	& \Leftrightarrow & \alpha \|k\|_{J}^2 \leq \sum\limits_{n \in \sigma} |[k, J z_n]_{J}|^2 + \sum\limits_{n \in \sigma^c} |[k, Jz'_n]_{J}|^2 \leq \beta \|k\|_{J}^2.
\end{eqnarray*}
This completes the proof.\\

In the following result, we present necessary and sufficient conditions for the existence of Krein space weaving frames. 

\begin{prop}
Let $\{z_n\}_{n=1}^\infty$ and $\{z'_n\}_{n=1}^\infty$ be frames for the Krein space $(\mathcal K, [.,.])$ with the bounds $\alpha_1 \leq \beta_1$ and $\alpha_2 \leq \beta_2$, then the following are equivalent:
\begin{enumerate}
\item $\{z_n\}_{n=1}^\infty$ and $\{z'_n\}_{n=1}^\infty$ are weaving frames for $\mathcal K$.

\item For every $\sigma \subset \mathbb N$, if $S_{\sigma}$ is the corresponding weaving frame operator, then for every $k \in \mathcal K$ there exists $\alpha >0$ so that we have,
$$\|S_{\sigma}(k)\|_{J} \geq \alpha \|k\|_{J}.$$
\end{enumerate}
\end{prop}

\begin{proof}  \noindent  (\underline{1 $\implies$ 2})
	
Let $\{z_n\}_{n=1}^\infty$ and $\{z'_n\}_{n=1}^\infty$ be weaving frames for $(\mathcal K, [.,.])$ with the universal bounds $\alpha, \beta$. Then for every $k \in \mathcal K$ and every $\sigma \subset \mathbb N$ we have,
\begin{equation*}
\alpha \|k\|_{J}^2 \leq \sum\limits_{n \in \sigma} |[k, z_n]|^2 + \sum\limits_{n \in \sigma^c} |[k, z'_n]|^2 \leq \beta \|k\|_{J}^2.
\end{equation*}
and hence applying theorem \ref{char} we obtain,
\begin{equation}\label{2}
\alpha \|k\|_{J}^2 \leq \sum\limits_{n \in \sigma} |[k, z_n]_{J}|^2 + \sum\limits_{n \in \sigma^c} |[k, z'_n]_{J}|^2 \leq \beta \|k\|_{J}^2.
\end{equation}
Thus using equation (\ref{2}), for every $k \in \mathcal K$ and every $\sigma \subset \mathbb N$ we obtain,
\begin{eqnarray*}
\|S_{\sigma}k\|_{J} = \sup\limits_{\|l\|=1} |[S_{\sigma}k, l]_{J}| & \geq & \left [ S_{\sigma}k, \frac{k}{\|k\|_{J}}\right ]_{J}\\
& = & \frac{1} {\|k\|_{J}} [S_{\sigma}k, k]_{J}\\
& \geq & \alpha \|k\|_{J}.
\end{eqnarray*}

 \noindent  (\underline{2 $\implies$ 1}) For every $k \in \mathcal K$ and every $\sigma \subset \mathbb N$, the corresponding weaving frame operator is given by, $S_{\sigma} k = TJT^* k$, where $T, T^*$ and $J$ are associated synthesis, analysis operators and fundamental symmetry on $(\mathcal K, [.,.])$. 
 
 Since $\|J\| = 1$ and $\|S_{\sigma}(k)\|_{J} \geq \alpha \|k\|_{J}$, and $\{z_n\}_{n=1}^\infty$ and $\{z'_n\}_{n=1}^\infty$ are frames for the Krein space $(\mathcal K, [.,.])$ with the bounds $\alpha_1 \leq \beta_1$ and $\alpha_2 \leq \beta_2$ then we have,
 \begin{equation}\label{3}
 \alpha^2 \|k\|_{J}^2 \leq \|T\|_{J}^2 \|T^*k\|_{J}^2 \leq \|T\|_{J}^2 (\beta_1 + \beta_2) \|k\|_{J}^2.
\end{equation}  
Again we have,
\begin{equation}\label{4}
\|T\|_{J}^2 = \|T^*\|_{J}^2 = \sup\limits_{k \neq 0} \frac{\|T^*k\|_{J}^2} {\|k\|_{J}^2} \leq \frac { (\beta_1 + \beta_2) \|k\|_{J}^2}{\|k\|_{J}^2} = (\beta_1 + \beta_2).
\end{equation}
Thus using equations (\ref{3}) and (\ref{4}), for every $k \in \mathcal K$ and every $\sigma \subset \mathbb N$ we obtained,
\begin{equation}
\frac{\alpha^2}{(\beta_1 + \beta_2)} \|k\|_{J}^2 \leq \sum\limits_{n \in \sigma} |[k, z_n]_{J}|^2 + \sum\limits_{n \in \sigma^c} |[k, z'_n]_{J}|^2 .
\end{equation}
The upper weaving frame bound will be achieved automatically under the given assumption. Therefore, applying theorem \ref{char}, $\{z_n\}_{n=1}^\infty$ and $\{z'_n\}_{n=1}^\infty$ are weaving frames for $\mathcal K$.
\end{proof}

The validity of the above theorem is illustrated through the following example.

\begin{ex}
Let us consider the frames 
$$
\{z_n\}_{n=1}^\infty=\{e_1, e_2, e_1, e_2, e_3, e_3, \cdots\}
$$ and 
$$
\{z'_n\}_{n=1}^\infty=\{e_1, e_1, e_2, e_2, e_3, e_3, \cdots\}
$$ for $\ell^2=\mathcal K$. Then for every $\sigma \subset \mathbb{N}$, it is straightforward to verify that the necessary and sufficient conditions of the theorem are satisfied.
\end{ex}

The following result establishes the wovenness of a Krein space frame and its image under a bounded linear operator, subject to certain conditions.
 
 \begin{prop} \label{2.8}
Let us consider a frame $\{z_n\}_{n=1}^\infty$  for Krein sapce $(\mathcal K, [.,.])$ with bounds  $\alpha$,  $\beta$ and $U$ be a bounded operator. If  $\|I - U\|_{J}^2 < \frac{\alpha}{\beta}$ then $\{z_n\}_{n=1}^\infty$ and $\{Uz_n\}_{n=1}^\infty$ are weaving frames.
\end{prop}
 
 \begin{proof}
 Clearly $U$ is invertiable and $\{Uz_n\}_{n=1}^\infty$ is a frame for $\mathcal K$.
 For every $k \in \mathcal K$ we have,      
 \begin{align*}
 \left(\sum\limits_{n \in \sigma} |[k, z_n]|^2 + \sum\limits_{n \in \sigma^c} |[k, Uz_n]|^2 \right)^\frac{1}{2} \\
 &\hspace*{-3cm} = \left(\sum\limits_{n\in\sigma} |[k, z_n]|^2 + \sum\limits_{n \in \sigma^c} |[k, z_n] - [(I - U^*)k, z_n]|^2 \right)^\frac{1}{2}\\
 & \hspace*{-3cm} \geq  \left(\sum\limits_{n \in \mathbb{N}} |[k, z_n] - \chi_{\sigma^c}(n)[(I - U^*)k, z_n]|^2 \right)^\frac{1}{2}\\
  & \hspace*{-3cm} \geq \left (\sum\limits_{n \in \mathbb{N}} |[k, z_n]|^2\right)^\frac{1}{2} - \left(\sum\limits_{n \in \mathbb{N}}|\chi_{\sigma^c}(n)[(I - U^*)k, z_n]|^2 \right)^\frac{1}{2}\\
  & \hspace*{-3cm} = \left (\sum\limits_{n \in \mathbb{N}} |[k, z_n]|^2\right)^\frac{1}{2} - \left(\sum\limits_{n \in \sigma^c}(|[(I - U^*)k, z_n]|^2 \right)^\frac{1}{2}\\
  & \hspace*{-3cm} \geq \sqrt{\alpha} \|k\|_{J} - \sqrt{\beta} \|(I - U^*)k\|_{J}\\
  & \hspace*{-3cm} \geq  (\sqrt{\alpha} - \sqrt{\beta} \|(I - U^*)\|_{J})\|k\|_{J}.
 \end{align*}

 

Again for every $k \in \mathcal K$ we have, 
 \begin{eqnarray*}
 \sum\limits_{n \in \sigma} |[k, z_n]|^2 + \sum\limits_{n \in \sigma^c} |[k, Uz'_n]|^2 
 &\leq&  \sum\limits_{n \in \mathbb{N} } |[k, z_n]|^2 + \sum\limits_{n \in \mathbb{N}} |[k, Uz'_n]|^2
 \\ &\leq& ( \beta + \|U\|^2) \|k\|^2.
 \end{eqnarray*} 
Thus $\{z_n\}_{n=1}^\infty$ and $\{Uz_n\}_{n=1}^\infty$ are weaving frames.
\end{proof}

\begin{ex}

  Let $k_1 = (1,0),$ and $k_2 = (0, 2) $ be two elements in the Krein sapce $\mathcal{K}=(\mathbb{R}^{2}, [., . ])$. Let us define the indefinite inner product as  $[x, y] = [(x_1, x_2), (y_1, y_2)] = x_1y_1 - x_2y_2$. The set $\{k_i\}^2_{i = 1}$  forms a frame in the krein space with bounds ($\alpha$ = 1 )and ($\beta$ = 2).
				
  Suppose $U$ is a bounded linear operator given by $U(x_1, x_2) = (x_1, 0.8x_2)$ and compute $\|I - U\|_{J} = 0.2$. Since $\|I - U\|^2_{J} = 0.04 < (\dfrac{\alpha}{\beta} = \dfrac{1}{2})$, then the Proposition \ref{2.8} ensures that $\{k_i\}^2_{i = 1}$ and $\{Uk_i\}^2_{i =1}$  are weaving frames.
	\end{ex}

\begin{ex}
	
Here we discuss a counter example of the Proposition \ref{2.8}.
	 Let us consider two elements $k_1 = (1,0),$ and $k_2 = (0, 3) $ in the Krein sapce $\mathcal{K}=(\mathbb{R}^{2}, [., . ])$ and the corresponding indefinite inner product is defined by $[x, y] = [(x_1, x_2), (y_1, y_2)] = x_1y_1 - x_2y_2$. Then the set $\{k_1, k_2\}$ forms a frame for the Krein space with bounds ($\alpha$ = 1 )and ($\beta$ = 3).
	
  Let $U$ be a bounded and  linear operator given by $U(x_1, x_2) = (-x_1, x_2)$, clearly $\{Uk_i\}^2_{i =1}$  is a frame. It is easy to verify that  $\{k_i\}^2_{i = 1}$ and $\{Uk_i\}^2_{i =1}$  are weaving frames.
However $\|I - U\|^2_{J} = 4 > (\dfrac{\alpha}{\beta} = \dfrac{1}{3})$.
	
\end{ex}

The following result establishes a condition for weaving frames through their respective frame operators.

\begin{prop}
	
	Let  $K$ = $\{z_n\}^\infty_{n =1}$ and $K'$ = $\{z'_n\}^\infty_{n= 1}$ be weaving frames  with universal frame bounds $\alpha$, $\beta$ and the corresponding frame operators  $S_K$ and $S_K'$, respectively .
	If  $\|S_K\|_J \|S_K^{-1} - S^{-1}_{K'}\|_J < \dfrac{\alpha}{\beta}$ (or $\|S_K'\|_J \|S_K^{-1} - S^{-1}_{K'}\|_J < \dfrac{\alpha}{\beta}$), then $S^{-1}_K = {\{S^{-1}_kz_n}\}^\infty_{n= 1}$ and $S^{-1}_{K'} = {\{S^{-1}_{k'}z'_n}\}^\infty_{n= 1}$
	are  weaving frames.
	
\end{prop}

\begin{proof}
	Without loss of generality let us  consider  $\|S_K\|_J \|S_K^{-1} - S^{-1}_{K'}\|_J < \dfrac{\alpha}{\beta}$. 
	For every $\sigma \subset \mathbb{N}$ and every $k \in (\mathcal{K}, [., .]$) we have,
	\begin{eqnarray*}
		&& \left (\sum\limits_{n \in \sigma} |[k, S^{-1}_Kz_n]|^2 + \sum\limits_{n \in \sigma^c} |[k, S^{-1}_{K'}z'_n]|^2 \right)^\frac{1}{2} \\
		&=& \left(\sum\limits_{n\in\sigma}|[S^{-1}_Kk, z_n]|^2 + \sum\limits_{n \in \sigma^c} |[S^{-1}_{K'}k, z'_n]|^2\right)^\frac{1}{2}\\ 
		&\geq& \left(\sum\limits_{n\in\sigma}|[S^{-1}_Kk, z_n]|^2 + \sum\limits_{n \in \sigma^c} |[S^{-1}_{K'}k +(S^{-1}_{K'} - S^{-1}_K)k, z'_n]|^2\right)^\frac{1}{2}
	\\	&\geq& \sqrt{\alpha} \|S^{-1}_Kk\| - \sqrt{\beta} \|(S^{-1}_{K'} - S^{-1}_K)k\|
\end{eqnarray*}
		On the other hand we have,
	\begin{eqnarray*}
			\\&&\hspace{-1cm}\sum\limits_{n \in \sigma} |[k, S^{-1}_Kz_n]|^2 + \sum\limits_{n \in \sigma^c} |[k, S^{-1}_{K'}z'_n]|^2 \\&\leq& \sum\limits_{n \in \mathbb{N}} |[S^{-1}_Kk, z_n]|^2 + \sum\limits_{n \in \mathbb{N}} |[S^{-1}_{K'}k, z'_n]|^2 \\&\leq& \sqrt{\beta}(\|S^{-1}_K\|^2 + \|S^{-1}_{K'}\|^2)\|k\|^2.
	\end{eqnarray*}
	This completes the proof.
%
%
	
\end{proof}

In the following theorem, we characterize weaving frames in Krein space using the projection operator.

\begin{thm}
	
	Let $\{z_n\}_{n=1}^\infty$ and $\{z'_n\}_{n=1}^\infty$ be two frames for Krein space $\mathcal{K},[., .]$, then the following are equivalent:
	
	\begin{enumerate}
	
	\item  The frames $\{z_n\}_{n=1}^\infty$ and $\{z'_n\}_{n=1}^\infty$ are weaving frames for $\mathcal{K}$.
	
\item	 For every $\sigma\subset \mathbb{N}$, let $Z_\sigma = \overline{span}\{z_n\}_{n \in\sigma}$,  $Z_{\sigma^{c}}= \overline{span}\{z'_n\}_{n \in\sigma^{c}}$ and let $P$ be the orthogonal projection of $\mathcal{K}$ onto $Z^\perp_{\sigma}$. Then $P|Z_{\sigma^{c}}$ is the orthogonal projection onto $Z^\perp_\sigma$ and $\{Pz'_n\}_ {n \in \sigma^c}$ is a frame for $Z^\perp_{\sigma}$.
		\end{enumerate}
\end{thm}

\begin{proof}
	
 \noindent  (\underline{1 $\implies$ 2})\\
 First we prove that for every $\sigma \subset \mathbb N$, $P|Z_{\sigma^c}$ be the orthogonal projection  onto $Z^\perp_{\sigma^c}$.
 
	Let us define $ Z_{\sigma} = \overline{span} \{z_n\}_{n \in \sigma}$, 
	$ Z_{\sigma^c} = \overline{span} \{z'_n\}_{n \in \sigma^c}$ and P be the orthogonal projection of $\mathcal{K}$ onto $Z_{\sigma}^\perp$.
	Assumed condition implies that for every $\sigma \subset \mathbb N$, $(\{z_n\}_{n \in \sigma} \cup \{z'_n\}_{n \in \sigma^c})$ is a  frame for $\mathcal{K}$.
	Since $\{z'_n\}_{n \in \sigma^c}$ spans $Z_{\sigma}^\perp$ then every $k \in Z_{\sigma}^\perp$ can be written as linear combination of $\{z'_n\}_{n \in \sigma^c}$, therefore  $P|Z_{\sigma^c}$ be the orthogonal projection  onto $Z^\perp_{\sigma}$.
	

	Since $(\{z_n\}_{n \in \sigma} \cup \{z'_n\}_{n \in \sigma^c})$ is a  frame for $\mathcal{K}$, then for every  $k \in \mathcal{K}$, there are universal constants $\alpha \leq \beta$ we have,
	$$\alpha \|k\|_{J}^2 \leq \sum\limits_{n \in \sigma} |[k, z_n]|^2 + \sum\limits_{n \in \sigma^c} |[k, z'_n]|^2 \leq \beta \|k\|_{J}^2.$$
Hence	for every $k \in Z_{\sigma}^\perp$ we have,
	$$\alpha \|k\|_{J}^2 \leq  \sum\limits_{n \in \sigma^c} |[k, z'_n]|^2 \leq \beta \|k\|_{J}^2.$$
	Since $P(z'_n) = z'_n$, then $\{Pz'_n\}_{n \in \sigma^c}$ is a frame for $Z_{\sigma}^\perp.$\\
	
	\noindent  (\underline{2 $\implies$ 1})\\
	Since $\{z_n\}_{n=1}^\infty$ is a frame  for $\mathcal{K}$ and $ Z_\sigma = \overline{span}(\{z_n\}_{n\in \sigma})$) then for every $\sigma \subset\mathbb{N}$,
	$\{z_n\}_{n\in \sigma}$ is frame for $Z_\sigma.$
	Therefore, for every $\sigma\subset \mathbb{N}$ and for every $k_\sigma \in Z_\sigma$, there exist $0<\alpha_1\leq \beta_1< \infty$ so that, 
	\begin{equation}\label{12}
	\alpha_1 \|k_\sigma\|_{J}^2 \leq \sum\limits_{n \in \sigma} |[k_\sigma, z_n]|^2  \leq \beta_1 \|k_\sigma\|_{J}^2.
	\end{equation}
	Moreover, $\{Pz_n\}_{n \in \sigma^c}$ is a frame for $Z^\perp_{\sigma}.$ Therefore, for every $\sigma\subset \mathbb{N}$ and for every $k^\perp_{\sigma} \in Z^\perp_\sigma$, there exist $0<\alpha_2\leq \beta_2< \infty$ so that
	\begin{equation}\label{13}
		\alpha_2\|k_\sigma^\perp\|_{J}^2 \leq \sum\limits_{n \in \sigma^c} |[k_\sigma^\perp, Pz'_n]|^2  \leq \beta_2 \|k_\sigma^\perp\|_{J}^2.
	\end{equation}
		Furthermore, since $P|Z_{\sigma^c}$ is the orthogonal projection  onto $Z^\perp_{\sigma}$, then from the equation (\ref{13}) we have,
		\begin{equation}\label{14}
		\alpha_2\|k_\sigma^\perp\|_{J}^2 \leq \sum\limits_{n \in \sigma^c} |[k_\sigma^\perp, z'_n]|^2  \leq \beta_2 \|k_\sigma^\perp\|_{J}^2.
	\end{equation}
	Again for every $k \in \mathcal{K}$ we have, $k = k_\sigma + k^\perp_\sigma$.
	Thus applying equations (\ref{12}) and (\ref{14}), for every $\sigma\subset \mathbb{N}$ we obtain,
	
	$$\alpha^*(\|k_\sigma^\perp\|_{J}^2 +\|k_\sigma^\perp\|_{J}^2 )\leq \sum\limits_{n \in \sigma} |[k_\sigma, z'_n]|^2 +\sum\limits_{n \in \sigma^c} |[k_\sigma^\perp, z'_n]|^2  \leq \beta^* (\|k_\sigma^\perp\|_{J}^2 +\|k_\sigma^\perp\|_{J}^2),$$
	where $\alpha^* = \min (\alpha_1, \alpha_2)$	and $\beta^* = \beta_1 + \beta_2$. Consequently, for every $\sigma\subset \mathbb{N}$ and for every $k \in \mathcal{K}$ we obtain,
	
	$$\alpha^*\|k\|_{J}^2  \leq \sum\limits_{n \in \sigma} |[k_\sigma, z'_n]|^2 +\sum\limits_{n \in \sigma^c} |[k_\sigma^\perp, z'_n]|^2  \leq \beta^* \|k\|_{J}^2.$$
This completes the proof.	
\end{proof}

\begin{defn}\cite{Gi12}\label{Bessel} 
A Bessel sequence $\{z_n\}^\infty_{n = 1}$ is said to be a $J$-frame for ($\mathcal{K}, [., .]$) if $R(T_{\sigma}P^+)$ is maximal uniformly $J$-positive subspace of $\mathcal{K}$ and $R(T_{\sigma}P^-)$ is maximal uniformly $J$-negative subspace of $\mathcal{K}$.

	  \end{defn}

\begin{rem}\label{rem}
	Let ($\mathcal{K}, [., .]$) be a Krein space with fundamental symmetry $J$. Consider $\{z_n\}^\infty_{n = 1}$ be a Bessel sequence in $\mathcal{K}$. Let 
	$$T_{\sigma}: \ell^2(\mathbb N)\rightarrow \mathcal{K},$$ be the associated synthesis operator. Define the index sets by 
	$$\sigma_+ =\{n \in\mathbb{N} : [z_n, z_n]\geq0\}\quad and\quad \sigma_- =\{n \in\mathbb{N} : [z_n, z_n]< 0\},$$ then the Hilbert space $\ell^2(\mathbb N)$ can be orthogonally decomposed as 
	$$\ell^2(\mathbb N) = \ell^2(\sigma_+) \oplus \ell^2(\sigma_-).$$
	Let $P^+$ and $P^-$ be the orthogonal of projections from $\ell^2(\mathbb N)$ onto the subspaces $\ell^2(\sigma_+)$ and $\ell^2(\sigma_-)$, respectively. If $M_{\sigma_+} = \overline{span} \{z_n\}_{n \in \sigma_+}$ and $M_{\sigma_-} = \overline{span} \{z_n\}_{n \in \sigma_-},$
	then it is easy to check that $\overline{span} \{z_n\}_{n \in \sigma_+} \subset R(T_{\sigma}P^+) \subset M_{\sigma_+},$
	$\overline{span} \{z_n\}_{n \in \sigma_-} \subset R(T_{\sigma}P^-) \subset M_{\sigma_-} \quad.$ Thus we have, 
	$$R(T_\sigma )= R(T_{\sigma}P^+) +R(T_{\sigma}P^-).$$ 
	
\end{rem}

We discuss $J$-weaving frames in Krein spaces, examining their structural properties and implications.

\begin{defn}\label{J-weaving}
	Let $\{z_n\}^\infty_{n= 1}$ and $\{z'_n\}^\infty_{n= 1}$ be two $J$-frames for Krein space ($\mathcal{K}, [., .]$) then they are said to be $J$-weaving frames if for every subset $\sigma \subset{\mathbb{N}}$,  $\{z_n\}_{n \in \sigma} \cup \{z'_n\}_{n \in \sigma^c}$ is a $J$-frame for $\mathcal{K}$.
	
If $\{z_n\}^\infty_{n= 1}$ and $\{z'_n\}^\infty_{n= 1}$ are $J$-weaving frames then for every subset $\sigma \subset\mathbb{N}, \{z_n\}_{n \in \sigma}  \cup  \{z'_n\}_{n \in \sigma^c}$ is a $J$-frame, then due to its maximality the range is closed. Thus applying \cite[Corollary 1.5.2]{An79} we have, $$R(T_{\sigma}P^+) = M_{\sigma_+} \quad and \quad R(T_{\sigma}P^-) = M_{\sigma_-} ,$$ and hence
	$$R(T_{\sigma}P^+) +R(T_{\sigma}P^-) = M_{\sigma_+} + M_{\sigma_-} = \mathcal{K}.$$ It is to be noted that $M_{\sigma_+}  = \overline{span} (\{z_n\}_{n \in \sigma_+} \cup \{z'_n\}_{n \in \sigma^c_+})$ and $M_{\sigma_-}  = \overline{span} (\{z_n\}_{n \in \sigma_-} \cup \{z'_n\}_{n \in \sigma^c_-})$.
\end{defn}

	

\begin{ex}
	If we consider $$
	\{z_n\}_{n =1}^3 = \left\{ \left(\dfrac{1}{2.16}, 0, 0 \right), \left(0, 0, \dfrac{2}{2.17} \right), \left(0, \dfrac{1}{2.18}, 0 \right)\right\}
	$$ and 
$$
\{z'_n\}_{n =1}^3 = \left\{  \left(\dfrac{100}{\sqrt{2}}, 0, \dfrac{101}{\sqrt{2}}\right), \left(\dfrac{102}{\sqrt{3}}, 0, \dfrac{103}{\sqrt{2}}\right), \left(0, \dfrac{1}{2.19}, 0\right)\right\},
$$ then it is easy to verify that $\{z_n\}_{n =1}^3$ and $\{z'_n\}_{n =1}^3$ are J-weaving frames for Krein space $\mathcal{K} = (\mathbb{R}^3, [., .])$ where  $$[(x_1, x_2, x_3), (y_1, y_2, y_3)] = x_1y_1 - x_2 y_2 + x_3 y_3.$$ 

Using the associated index sets in Remark \ref{rem}, we have the following cases:
 
{\bf Case 1:} $\sigma_+ = \{1\}$, $ M_{\sigma_+} = \overline{span} (\{z_1\} \cup \{z'_2\})$ is a maximal uniformly $J$-positive subspace of $\mathcal{K}.$ 
 
 {\bf Case 2:} $\sigma_+ = \{2\}$, $ M_{\sigma_+} = \overline{span} (\{z_2\} \cup \{z'_1\})$ is a maximal uniformly $J$-positive subspace of $\mathcal{K}.$
	
{\bf Case 3:} $\sigma_+ = \{1, 2\}$,  $ M_{\sigma_+} = \overline{span} \{z_n\}_{n = 1}^2 $ is a maximal uniformly $J$-positive subspace of $\mathcal{K}.$
	 
{\bf Case 4:}	Otherwise,  $ M_{\sigma_+} = \overline{span} \{z'_n\}_{n = 1}^2 $ is a maximal uniformly $J$-positive subspace of $\mathcal{K}.$ 
	  
It is to be noted that for every subset $\sigma_+,  M_{\sigma_+} = \overline{span} (\{z_n\}_{n \in   \sigma_+} \cup \{z'_n\}_{n \in \sigma^c_+}),$ is a maximal uniformly $J$-positive subspace of $\mathcal{K}.$ 
	  
Analogously for every subset $\sigma_-,  M_{\sigma_-} = \overline{span} (\{z_n\}_{n \in   \sigma_-} \cup \{z'_n\}_{n \in \sigma^c_-}),$ is a maximal uniformly $J$-negative subspace of $\mathcal{K}.$

\end{ex}

The following is a non-example of the definition \ref{J-weaving}, illustrating that not every weaving frames in Krein space necessarily qualify as  $J$-weaving frames.

\begin{ex}

	
	Let $(\mathcal{K} = \mathbb{C}^3, [., .])$ be a Krein space and define the associated indefinite inner product as $$[(x_1, x_2, x_3), (y_1, y_2, y_3)] = x_1\overline y_1 + x_2\overline y_2 - x_3\overline y_3.$$
Let us consider two frames 
$$ 
\{z_n\}_{n =1}^3 = \left\{ \left(3, 0, \dfrac{3}{\sqrt{2}}\right), \left(0, \dfrac{2}{\sqrt{3}}, 0\right), \left(0, 0, \dfrac{1}{\sqrt{3}}\right)\right\}
$$ and 
$$
\{z'_n\}_{n =1}^3 = \left\{ \left(1, 0, 0\right), \left(0, 3, \dfrac{3}{\sqrt{2}}\right), \left(0, 0,1\right )\right\}
$$ in $\mathbb C^3$. Then it is easy to verify that they are weaving frames, however they are not $J$-weaving frames.
	
	Here  $M_{\sigma_+}  = \overline{span} (\{z_n\}_{n \in \sigma_+} \cup \{z'_n\}_{n \in \sigma^c_+})$, $M_{\sigma_-}  = \overline{span} (\{z_n\}_{n \in \sigma_-} \cup \{z'_n\}_{n \in \sigma^c_-})$ where $\sigma_+ =\{n \in \mathbb{N}  : [z_n, z_n]\geq0\}$, $\sigma_- =\{n\in\mathbb{N} : [z_n, z_n]< 0\}$,  
	$\sigma^c_+ =\{n\in\mathbb{N}  : [z'_n, z'_n]\geq0\}$ and $\sigma^c_- = \{n \in\mathbb{N}  : [ z'_n, z'_n]< 0\}$. Thus $M_{\sigma_+}$ is not uniformly $J$-positive, whenever we consider $\sigma_+ = \{1\}$. Therefore, $\{z_n\}_{n = 1}^3$ and $\{z'_n\}_{n = 1}^3$ are not $J$-weaving frames.
	
\end{ex}

\begin{prop}\label{2.18}
If $\{z_n\}^\infty_{n= 1}$ and $\{z'_n\}^\infty_{n= 1}$ are $J$-weaving frames on Krein space $\mathcal{K}$ then for every $ \sigma_{\pm} \subset{\mathbb{N}}$ and for every $k \in M_{\sigma_\pm}$ there exist constants $ \beta_{-} \leq  \alpha_- < 0 < \alpha_+ \leq \beta_{+}$ we have, 


\begin{equation}\label{19}
	\alpha_{\pm} [k, k]  \leq \sum\limits_{n \in \sigma_{\pm}} |[k, z_n]|^2 + \sum\limits_{n \in \sigma_{\pm}^c} |[k, z'_n]|^2 \leq \beta_{\pm} [k, k].
\end{equation}

\end{prop}




\begin{proof}
Given $\{z_n\}^\infty_{n= 1}$ and $\{z'_n\}^\infty_{n= 1}$ are $J$-weaving frames, which impiles that for every $ \sigma \subset{\mathbb{N}}$, 
$\{z_n\}_{n \in \sigma} \cup \{z'_n\}_{n \in \sigma^c}$ is a $J$-frame for $\mathcal{K}$.
Thus $R(T_{\sigma}P^+) = M_{\sigma_+}$ is a maximal uniformly $J$-positive subspace for $\mathcal{K}$. Therefore, $T_{\sigma}P^+: \ell^2(\mathbb N)\rightarrow (M_{\sigma_+}, [., .])$ is an onto map. Hence for every subset $\sigma_+ \subset \mathbb{N},  \{z_n\}_{n \in \sigma_+} \cup \{z'_n\}_{n \in \sigma_+^c}$ is a frame for ($M_{\sigma_+}, [., .]$). In particular, there exist constants $0<\alpha_+ \leq \beta_{+}$ such that equation \ref{19} holds for $M_{\sigma_+}$. A similar assertion applies
for every $\sigma_- \subset \mathbb{N}.$
\end{proof}

The following result establishes a characterization theorem for $J$-weaving frames. This characterization provides deeper insights into their structural properties within Krein spaces.

\begin{thm}\label{2.20}
	Let $\{z_n\}^\infty_{n= 1}$ and $\{z'_n\}^\infty_{n= 1}$ be weaving frames for $\mathcal{K}, [., .]$. Define the index sets $\sigma_+ =\{n \in\mathbb{N} : [z_n, z_n]\geq0\}\quad and\quad \sigma_- =\{n \in\mathbb{N} : [z_n, z_n]< 0\}$. Suppose $M_{\sigma_+}  = \overline{span} (\{z_n\}_{n \in \sigma_+} \cup \{z'_n\}_{n \in \sigma^c_+})$ and  $M_{\sigma_-}  = \overline{span} (\{z_n\}_{n \in \sigma_-} \cup \{z'_n\}_{n \in \sigma^c_-})$. Then the following are equivalent:
	\begin{enumerate}
	\item $\{z_n\}^\infty_{n= 1}$ and $\{z'_n\}^\infty_{n= 1}$ are $J$-weaving frames for $\mathcal{K}, [., .]$.
	
\item If $M_{\sigma_+} \cap M^\perp_{\sigma_+} = \{0\}\quad and\quad M_{\sigma_-} \cap M^\perp_{\sigma_-} = \{0\} $, then for every $ \sigma_{\pm} \subset \mathbb{N}$ and for every $k \in M_{\sigma_{\pm}}$, there exist constants $ \beta_{-} \leq  \alpha_- < 0 < \alpha_+ \leq \beta_{+}$ we have,
	\begin{equation}	\alpha_{\pm} [k, k]  \leq \sum\limits_{n \in \sigma_{\pm}} |[k, z_n]|^2 + \sum\limits_{n \in \sigma_{\pm}^c} |[k, z'_n]|^2 \leq \beta_{\pm} [k, k].\end{equation}
	\end{enumerate}
	
\end{thm}

\begin{proof}
	 \noindent  (\underline{1 $\implies$ 2})\\
	 Suppose $\{z_n\}^\infty_{n= 1}$ and $\{z'_n\}^\infty_{n= 1}$ are $J$-weaving frames for $\mathcal({K}, [., .]$), then using the definition of $J$-weaving frames and applying Proposition \ref{2.18}, we achieve our desired result.
	
	\noindent  (\underline{2 $\implies$ 1})\\
	If  $M_{\sigma_+} \cap M^\perp_{\sigma_+} = \{0\}$ and $M_{\sigma_-} \cap M^\perp_{\sigma_-} = \{0\}$, then both $M_{\sigma_+}$ and $M_{\sigma_-}$ are $J$-non-degenerated. Then applying (Proposition 3.8. \cite{Gi12}),  for every $ \sigma_{\pm} \subset \mathbb{N}$ and for every $k \in M_{\sigma_{\pm}}$, there exist constants $ \beta_{-} \leq  \alpha_- < 0 < \alpha_+ \leq \beta_{+}$ so that 
	$$	\alpha_{\pm} [k, k]  \leq \sum\limits_{n \in \sigma_{\pm}} |[k, z_n]|^2 + \sum\limits_{n \in \sigma_{\pm}^c} |[k, z'_n]|^2 \leq \beta_{\pm} [k, k].$$
	Furthermore, applying (Theorem 3.9. \cite{Gi12}) $M_{\sigma_+}$ is a maximal uniformly $J$-positive subspace of $\mathcal{K}$ and  $M_{\sigma_-}$ is a maximal uniformly $J$-negative subspace of $\mathcal{K}.$ Therefore, $\{z_n\}^\infty_{n= 1}$ and $\{z'_n\}^\infty_{n= 1}$ are $J$-weaving frames for $\mathcal{K}, [., .]$.
\end{proof}

 \section{Application to Probabilistic Erasure}
 
 In this section, we explore the applications of Krein space weaving frames in the context of probabilistic erasure, focusing on their stability in data recovery scenarios. Our discussion begins by introducing the fundamental concepts of the tensor product of elements in Krein spaces, which plays a crucial role in understanding the interaction of frame elements. This foundational discussion sets the stage for analyzing how weaving frames can be applied to mitigate the effects of erasure in probabilistic settings.
 
Let $\{z_n\}_{n=1}^\infty$ and $\{z'_n\}_{n=1}^\infty$ be weaving frames for $\mathcal K$ with the corresponding frame operators $S_1, S_2$, then for every $\sigma \subset \mathbb N$, the associated error operator is given by,
\begin{equation*}
E=\sum\limits_{n \in \sigma} \delta_n z_n \otimes S_1^{-1} z_n + \sum\limits_{n \in \sigma^c} \delta_n z'_n \otimes S_2^{-1} z'_n,
\end{equation*}
where $\delta_n$ is the standard dirac delta function. 
 
This error operator enables rapid representation while significantly minimizing computational cost. Its effectiveness is largely dependent on the selection of Krein space frames utilized in the encoding and decoding process. By optimizing these frame choices, the overall efficiency of the technique can be further enhanced.

In the following result, we demonstrate how weaving frames can effectively mitigate data loss in erasure scenarios.

\begin{thm}
Let $\{z_n\}_{n=1}^m$ and $\{z'_n\}_{n=1}^m$ be uniform tight weaving frames for the Krein space $\mathbb R^n$ with $\|z_n\|_{J} = \sqrt n = \|z'_n\|_{J}$. Suppose $\epsilon>0$ is very small number for which $\epsilon^2 \geq \frac{n}{m} \text{log} (n)$. Then for every $k \in \mathbb R^n$, there is an absolute constant $M>0$ for which we have, $$E\|\hat k-k\|_{J} \leq M \epsilon \|k\|_{J}.$$
\end{thm}

\begin{proof}
For every $k \in \mathbb R^n$ and every $\sigma \subset \{1, 2, \cdots, m\}$ we have,
\begin{equation*}
\tilde{E} k = \hat k-k = \frac{1}{m} \left [\sum\limits_{n \in \sigma} \epsilon_n [k, z_n]z_n + \sum\limits_{n \in \sigma^c} \epsilon_n [k, z'_n]z'_n \right],
\end{equation*}
where 
\begin{align*}
\hat k & = \frac{2}{m}  \left [\sum\limits_{n \in \sigma} [k, z_n]z_n + \sum\limits_{n \in \sigma^c} [k, z'_n]z'_n \right]\\ 
& =\frac{2}{m}  \left [\sum\limits_{n \in \sigma} \tilde{\delta_n}^{p_n} [k, z_n]z_n + \sum\limits_{n \in \sigma^c} \tilde{\delta_n}^{p_n} [k, z'_n]z'_n \right],
\end{align*}
 with
$$
\tilde{E}=\sum\limits_{n \in \sigma} \tilde{\delta_n}^{p_n} z_n \otimes z_n + \sum\limits_{n \in \sigma^c} \tilde{\delta_n}^{p_n} z'_n \otimes z'_n 
~\mbox{ and }~
\epsilon_n = 2\tilde{\delta_n}^{p_n} - 1
$$ is an independent Bernoulli variable that takes values $1$ and $-1$, each with probability of $\frac{1}{2}$ (see \cite{Pa99}).

Thus it is sufficient to prove that, for every $\sigma \subset \{1, 2, \cdots, m\}$,
\begin{equation*}
E\left \|  \frac{1}{m} \left [\sum\limits_{n \in \sigma} \epsilon_n z_n \otimes z_n + \sum\limits_{n \in \sigma^c} \epsilon_n  z'_n \otimes z'_n \right] \right \|_{J} \leq M \epsilon.
\end{equation*}
Applying [Lemma, \cite{Ru99}] for every $\sigma \subset \{1, 2, \cdots, m\}$, there is a positive constant $M$ we have 

\begin{align}\label{app1}
E\left \| \left [\sum\limits_{n \in \sigma} \epsilon_n z_n \otimes z_n + \sum\limits_{n \in \sigma^c} \epsilon_n  z'_n \otimes z'_n \right] \right \|_{J}
& \leq  M \sqrt{n \text{log}(n)}\\
\nonumber & \hspace{0.3cm} \left \| \left [\sum\limits_{n \in \sigma} z_n \otimes z_n + \sum\limits_{n \in \sigma^c}  z'_n \otimes z'_n \right] \right \|^{\frac{1}{2}}_{J}.
\end{align}

Furthermore, since $\{z_n\}_{n=1}^m$ and $\{z'_n\}_{n=1}^m$ are uniform tight weaving frames for   $\mathbb R^n$ with $\|z_n\|_{J} = \sqrt n = \|z'_n\|_{J}$, then for every $\sigma \subset \{1, 2, \cdots, m\}$ we have,
\begin{equation*}
\mathcal I_{\mathbb R^n} =  \frac{1}{m} \left [\sum\limits_{n \in \sigma}  z_n \otimes z_n + \sum\limits_{n \in \sigma^c}  z'_n \otimes z'_n \right], 
\end{equation*}
where $\mathcal I_{\mathbb R^n}$ is the identity operator on  $\mathbb R^n$.

Hence we obtain,
\begin{equation}\label{app2}
\left \| \left [\sum\limits_{n \in \sigma} z_n \otimes z_n + \sum\limits_{n \in \sigma^c}  z'_n \otimes z'_n \right] \right \|_{J} =m.
\end{equation}
Thus using the equations (\ref{app1}) and (\ref{app2}) we obtain,
\begin{equation*}
E\left \|  \frac{1}{m} \left [\sum\limits_{n \in \sigma} \epsilon_n z_n \otimes z_n + \sum\limits_{n \in \sigma^c} \epsilon_n  z'_n \otimes z'_n \right] \right \|_{J} \leq M \sqrt{\frac{n}{m}} ~\sqrt {\text{log}(n)}.
\end{equation*}
Therefore, we achieved our result due to the fact $\epsilon^2 \geq \frac{n}{m} \text{log} (n)$.
\end{proof}

\section*{Acknowledgment}
The authors acknowledge the Department of Mathematics at SRM University, {\it AP} - Andhra Pradesh, for providing an enriching academic environment to carry out the research. The second author expresses his gratitude to his advisors, Prof. Saikat Mukherjee (NIT Meghalaya, India) and Prof. Jaydeb Sarkar (ISI Bangalore, India), for their valuable suggestions regarding his research.



\end{document}